\documentclass[11pt,reqno]{amsart}
\usepackage{amsmath,amsopn,amssymb,amsthm}
\usepackage[cp1251]{inputenc}
\usepackage[english]{babel}
\usepackage[pagebackref,breaklinks=true,colorlinks=true,linkcolor=blue,citecolor=blue,urlcolor=blue]{hyperref}
\usepackage{graphicx}%

\newtheorem{theorem}{Theorem}[section]

\newtheorem{corollary}[theorem]{Corollary}

\newtheorem{definition}[theorem]{Definition}

\newtheorem{lemma}[theorem]{Lemma}

\newtheorem{question}[theorem]{Question}

\renewenvironment{proof}[1][Proof]{\noindent\textbf{#1.} }{\ \rule{0.5em}{0.5em}}

\begin{document}
\title[Geodesic orbit Finsler space with $K\geq0$ and the (FP) condition]{Geodesic orbit Finsler space \\with $K\geq0$ and the (FP) condition}
\author{Ming Xu}

\address{Ming Xu \newline
School of Mathematical Sciences \newline
Capital Normal University \newline
Beijing 100048, P. R. China}
\email{mgmgmgxu@163.com}

\date{}
\maketitle

\begin{abstract}
In this paper, we study the interaction between the geodesic orbit (g.o.~in short) property and certain flag curvature conditions.
A Finsler manifold is called g.o.~if each constant speed geodesic is the orbit of a one-parameter subgroup.
Besides the non-negatively curved condition, we also concern the (FP) condition for the flag curvature, i.e., in any flag we can find a flag pole, such that the flag curvature  is positive. The main theorem we will prove is the following. If
a g.o.~Finsler space $(M,F)$ has non-negative flag curvature and satisfies the (FP) condition, then $M$ must be compact. Further more, if we present $M$ as $G/H$ where $G$ has a compact Lie algebra, then we have the rank inequality $\mathrm{rk}\mathfrak{g}\leq\mathrm{rk}\mathfrak{h}+1$. As an application of the main theorem, we prove that any even dimensional g.o.~Finsler space which has non-negative flag curvature and satisfies the (FP) condition must be a smooth coset space admitting positively curved homogeneous Riemannian or Finsler metrics.
\smallskip

\textbf{Mathematics Subject Classification (2010)}: 22E46, 53C22, 53C60.

\textbf{Key words}: flag curvature, geodesic orbit Finsler space, homogeneous Finsler space, homogeneous geodesic, non-negatively curved condition, (FP) condition
\end{abstract}

\section{Introduction}

A homogeneous Riemannian manifold is called a {\it g.o.~space}, i.e., a
{\it geodesic orbit space}, if any geodesic is the
orbit of a one-parameter subgroup of isometries.
This notion was introduced by O. Kowalski and L. Vanhecke
in 1991 \cite{KV1991}, as a generalization of the naturally
reductive homogeneity. Geodesic orbit Riemannian manifolds
have been studied extensively. See for example \cite{AA2007,AN2009,AV1999,BN2018,DKN2004,Go1996,GN2018}
and the references therein. In Finsler geometry, g.o.~space defined by S. Deng and Z. Yan \cite{YD2014} and
its subclasses, normal homogeneous Finsler spaces,
$\delta$-homogeneous Finsler spaces, etc., have also been studied in
recent years. See for example
\cite{XD2017-1,XD2017-2,Xu2018,XDY-preprint,XZ2018}.

Totally geodesic submanifold and Finslerian submersion provide important techniques for studying the flag curvature in homogeneous Finsler geometry \cite{XDHH,XZ2017}. For g.o.~Finsler spaces, these techniques can be refined to be more powerful (see \cite{XDY-preprint} and the discussion in Section \ref{sec-total-geo} and Section \ref{sec-sub-tech}). They help us explore the interactions between the g.o.~property and many different flag curvature properties.

The flag curvature properties we will discuss in this paper include not only the non-negatively curved condition, i.e., all flag curvatures are non-negative, but also the {\it (FP) condition} (or the
{\it flagwise positively curved condition}), i.e., in each flag we can find a flag pole such that the flag curvature is positive
(see Definition \ref{define-FP-condition} in Section 2).

The notion of the (FP) condition was introduced in \cite{Xu2017,XD2018}. Though it is equivalent to the positively curved condition when $F$ is Riemannian, generally speaking
it is much weaker in Finsler geometry \cite{Xu2017}. However, we believe that the combination of the non-negatively curved condition and the (FP) condition is an interesting
approximation to the positively curved condition, and some important geometric
or topological properties of positively curved homogeneous Finsler spaces \cite{DX-survey}
might be satisfied by flagwise positively and non-negatively
curved homogeneous Finsler spaces.

It is well known that the compactness (i.e. Bonnet-Myers theorem \cite{BCS2000}) and the rank inequality \cite{DH2013,Wa1972}
are among the most fundamental and important properties of positively curved homogeneous manifolds. So
we were enlightened to propose the following two questions (see Problem 4.1 and Problem 4.2 in \cite{XD2018}).

\begin{question}\label{main-question-1}
Let $(M,F)$ be a connected homogeneous Finsler manifold which
has non-negative flag curvature and satisfies the
(FP) condition. Can we prove the compactness for $M$?
\end{question}

\begin{question}\label{main-question-2}
Let $(G/H,F)$ be a connected homogeneous Finsler space which
has non-negative flag curvature and satisfies the
(FP) condition. Suppose $G$ is quasi-compact, i.e.
it has a compact Lie algebra. Can we prove the rank
inequality $\mathrm{rk}\mathfrak{g}\leq\mathrm{rk}\mathfrak{h}+1$?
\end{question}

In this paper, we will answer these two questions
from the positive side for g.o.~Finsler spaces, by the following theorem.

\begin{theorem}\label{main-thm}
Assume that $(M,F)$ with $\dim M>1$ is a connected g.o.~Finsler space which has
non-negative flag curvature and satisfies the (FP) condition, then $M$ is compact. If we present $M=G/H$ in which $G$ is quasi-compact, then we have
the rank inequality $\mathrm{rk}\mathfrak{g}\leq\mathrm{rk}\mathfrak{h}+1$,
where $\mathfrak{g}$ and $\mathfrak{h}$ are the Lie algebras
of $G$ and $H$ respectively.
\end{theorem}

When $(M,F)$ is normal homogeneous or $\delta$-homogeneous, the proof of Theorem \ref{main-thm} is much easier (see Theorem 4.2
in \cite{XZ2018}). But here we need more delicate and complicated usage of the Lie theories (Section \ref{sec-lie-theory}) and the techniques from totally geodesic submanifolds (Section \ref{sec-total-geo}) and Finslerian submersions (Section \ref{sec-sub-tech}), and the case by case
discussion for a special class of non-negatively curved g.o.~ Finsler spaces (Section \ref{sec-a-special-class-non-neg}).

As an application of Theorem \ref{main-thm}, we prove that
any even dimensional flagwise positively and non-negatively curved
connected g.o.~Finsler space is a compact coset space admitting
positively curved homogeneous Riemannian or Finsler metric (see
Corollary \ref{last corollary} at the end). Unfortunately, this is not a complete classification because we do not know if the three Wallach spaces,
$SU(3)/T^2$, $Sp(3)/Sp(1)^3$ and $F_4/Spin(8)$ \cite{Wa1972}, admit g.o.~Finsler metrics which have non-negative flag curvature and satisfy the
(FP) condition.

With the tools accumulated from Lie theory and geometry,
we expect more questions to be asked and studied for
non-negatively curved g.o.~Finsler spaces, or
more generally, for all non-negatively curved homogeneous Finsler
spaces (see Question \ref{Question-1} and Question \ref{Question-2} in Section \ref{sec-a-special-class-non-neg}).

This paper is organized as following. In Section 2, we summarize
some basic knowledge on Finsler geometry and Lie theory. In Section 3, we recall the notion of g.o.
Finsler space and discuss the totally geodesic and
submersion techniques for it. In Section 4, we consider a special
class of non-negatively curved g.o.~Finsler spaces as the preparation for proving Theorem \ref{main-thm}. In Section 5,
we prove Theorem \ref{main-thm} and its corollary.
\section{Preliminaries}
In this section, we summarize some fundamental knowledge on Finsler geometry and Lie theory.  Unless otherwise specified, in this paper we only discuss connected smooth manifolds which dimension is bigger than $1$.

\subsection{Finsler metric and Minkowski norm}

A {\it Finsler metric} on a manifold $M$ with $\dim M=m>0$ is a continuous function $F:TM\rightarrow [0,\infty)$ satisfying the
following conditions for any {\it standard} local coordinates $(x,y)\in TM$,
where $x=(x^i)\in M$ and $y=y^j\partial_{x^j}\in T_xM$:
\begin{enumerate}
\item The restriction of $F$ to $TM\backslash 0$ is a positive smooth function;
\item For any $\lambda\geq0$ and $y\in T_xM$,
$F(x,\lambda y)=\lambda F(x,y)$;
\item When $y\neq 0$, the Hessian matrix $(g^F_{ij}(x,y))=(\partial^2_{y^iy^j}[F^2]/2)$
    is positive definite.
\end{enumerate}
We will also call the pair $(M,F)$ a {\it Finsler manifold}
or a {\it Finsler space}.

The Hessian matrix $(g_{ij}^F(x,y))$ for a Finsler metric $F$
defines an inner product at the nonzero vector $y\in T_xM$, i.e.,
\begin{equation}\label{def-fundamental-tensor}
\langle u,v\rangle^F_y=u^iv^j g^F_{ij}(x,y)=
\frac12\frac{\partial^2}{\partial s\partial t}[F^2(y+su+tv)]|_{s=t=0},
\end{equation}
for any $u=u^i\partial_{x^i}$ and $v=v^j\partial_{x^j}$ in $T_xM$.
We call $\langle\cdot,\cdot\rangle_y^F $ the {\it fundamental tensor} and simply denoted it
as $g_y^F$ sometimes. A Finsler metric $F$ is {\it Riemannian} iff
$g_y^F$ for $y\in T_xM$ is only relevant to the point $x\in M$.

The restriction of a Finsler metric $F$ to each tangent space $T_xM$
is called a {\it Minkowski norm}. Minkowski norm can be abstractly defined on any real vector space $\mathbf{V}$ with $0<\dim\mathbf{V}<\infty$ by similar conditions as (1)-(3) above, and
we call the pair $(\mathbf{V},F)$ a {\it Minkowski space}.
The fundamental tensor for a
Minkowski norm can be similarly defined by (\ref{def-fundamental-tensor}).

See \cite{BCS2000,CS2005,Sh2001}
for more details.

\subsection{Flag curvature and the (FP) condition}

Let $(M,F)$ be a Finsler space, and consider any {\it flag triple} $(x,y,\mathbf{P})$ of it, where $x$ is a point on $M$, the {\it flag}
$\mathbf{P}$ is
a tangent plane in $T_xM$ and the {\it flag pole} $y$ is a nonzero vector in $\mathbf{P}$. Then the {\it flag curvature} for this flag triple
is defined by
$$K^F(x,y,\mathbf{P})=\frac{\langle R^F_y u,u\rangle_y^F}{
\langle y,y\rangle_y^F\langle u,u\rangle_y^F-
\langle y,u\rangle_y^F\langle y,u\rangle_y^F},$$
in which
$R^F_y:T_xM\rightarrow T_xM$ is the Riemann curvature
(see \cite{BCS2000} for its explicit formula by
standard local coordinates), and $u$ is any vector in $\mathbf{P}$
such that $\mathbf{P}=\mathrm{span}\{y,u\}$.

Generally speaking, the flag curvature in Finsler geometry depends
on the flag triple but not the vector $u$. When the metric is Riemannian, the flag curvature coincides with the sectional curvature and depend on $x$ and $\mathbf{P}$ only.

We call a Finsler space {\it positively curved} (or {\it non-negatively curved}) if the flag curvature for any triple
is positive (or non-negative, respectively).

By definition, flag curvature in Finsler geometry is a natural generalization of sectional curvature in Riemannian geometry. However, it is much more local a geometric quantity. So we may propose the following variation for the positively curved condition, which is called the {\it flagwise positively curved condition} or the {\it (FP) condition} \cite{XD2018}.

\begin{definition}\label{define-FP-condition}
We call a Finsler space $(M,F)$ flagwise positively curved or satisfying the (FP) condition, if for any $x\in M$ and any tangent plane $\mathbf{P}\subset T_xM$, we can find a nonzero vector $y\in \mathbf{P}$, such that $K^F(x,y,\mathbf{P})>0$.
\end{definition}

The trace of the Riemann curvature $R^F_y:T_xM\rightarrow T_xM$ is also
an important geometric quantity in Finsler geometry, called the Ricci scalar, and denoted as $\mathrm{Ric}(x,y)=\mathrm{Tr} R^F_y$. Using flag curvature, Ricci scalar can be presented as following.
If we take any $g^F_y$-orthonormal basis $\{u_1,\ldots,u_{m-1}\}$
for the $g^F_y$-orthogonal complement of a nonzero tangent vector $y$ in
$T_xM$, and the tangent planes $\mathbf{P}_i=\mathrm{span}\{y,u_i\}$, then we have
$$\mathrm{Ric}(x,y)=F^2(x,y)\sum_{i=1}^{m-1}K^F(x,y,\mathbf{P}_i).$$
So any non-negatively curved Finsler space must have non-negative Ricci scalar.

See \cite{BCS2000,CS2005,Sh2001} for more details.
\subsection{Finslerian Submersion}

In \cite{PD2001}, J.C. \'{A}lvarez Paiva and C.E. Dur\'{a}n introduced the notion of {\it Finslerian submersion}.

A linear endomorphism $l:(\mathbf{V}_1,F_1)\rightarrow(\mathbf{V}_2,F_2)$ between two
Minkowski spaces
is called a {\it linear submersion}, if $l$ maps the $F_1$-unit ball in $\mathbf{V}_1$ onto the $F_2$-unit ball in $\mathbf{V}_2$.

A smooth map $f:(M_1,F_1)\rightarrow (M_2,F_2)$ between two Finsler spaces is called a {\it Finslerian submersion} (or simply
{\it submersion}) if its tangent map $f_*:(T_x{M_1},F_1(x,\cdot))\rightarrow (T_{f(x)}M_2,F_2(f(x),\cdot)$ is a linear submersion everywhere
\cite{PD2001}.

When the submersion $f:(M_1,F_1)\rightarrow (M_2,F_2)$ is surjective, the metric $F_2$ can be uniquely determined by
$f$ and $F_1$. The existence of $F_2$ can be
insured by the following lemma sometimes
(see Lemma 2.4 in \cite{XDY-preprint}).

\begin{lemma}\label{lemma-0}
Assume that the Finsler space $(M,F)$ admits the isometric $G$-action for some Lie group $G$ and the orbit space $G\backslash M$
is a smooth manifold such that the natural projection $\pi:M\rightarrow G\backslash M$ is a smooth map. Then we can find a
unique Finsler metric $F'$ on $G\backslash M$, such that $\pi:(M,F)\rightarrow(G\backslash M,F')$ is a submersion.
\end{lemma}

Horizonal lifting is a crucial notion in the study of submersion.
For a linear submersion $l:(\mathbf{V}_1,F_1)\rightarrow
(\mathbf{V}_2,F_2)$, any vector $v_2\in\mathbf{V}_2\backslash\{0\}$ has a unique
{\it horizonal lifting} $v_1\in\mathbf{V}_2\backslash\ker l$ defined by $l(v_1)=v_2$ and $F_1(v_1)=F_2(v_2)$.

For a submersion between two Finsler spaces,
$f:(M_1,F_1)\rightarrow(M_2,F_2)$ with $f(x_1)=x_2$,
any smooth curve $c_2(t)$ with $c_2(0)=x_2$ on $(M_2,F_2)$
has a unique horizonal lifting at $x_1$, i.e., a smooth curve $c_1(t)$ with $c_1(0)=x_1$ on $(M_1,F_1)$, such that
for each $t$, $\dot{c}_1(t)$ is the horizonal lifting of $\dot{c}_2(t)$, for the linear submersion $f_*:
(T_{c_1(t)}M_1,F_1(c_1(t),\cdot))\rightarrow
(T_{c_2(t)}M_2,F_2(c_2(t),\cdot))$. It is an important observation
that the horizonal lifting of a constant speed geodesic is  geodesic with the same constant speed.

For a submersion $f:(M_1,F_1)\rightarrow(M_2,F_2)$, the {\it horizonal lifting} of a flag triple can be similarly defined.
Its importance is revealed by the following generalization for
the classical O'Neill formula \cite{PD2001}.

\begin{theorem}\label{theorem-flag-curvature-inequality-Finsler-submersion}
For any flag triple $(x_2,y_2,\mathbf{P}_2)$ of $(M_2,F_2)$ in the submersion $f:(M_1,F_1)\rightarrow(M_2,F_2)$, and any $x_1\in f^{-1}(x_2)$, horizonal lifting provides a unique flag triple
$(x_1,y_1,\mathbf{P}_1)$ of $(M_1,F_1)$, and we have
$K^{F_1}(x_1,y_1,\mathbf{P}_1)\leq K^{F_2}(x_2,y_2,\mathbf{P}_2)$.
\end{theorem}

We will frequently use the following immediate corollary of
Theorem \ref{theorem-flag-curvature-inequality-Finsler-submersion}.

\begin{lemma}\label{coro-flag-curv-submersion}
Let $f:(M_1,F_1)\rightarrow(M_2,F_2)$ be a surjective
Finslerian submersion. Suppose that $(M_1,F_1)$ has positive or
non-negative flag curvature, then so does $(M_2,F_2)$ respectively.
\end{lemma}

\subsection{Totally geodesic Finsler submanifold}

Let $M'$ be a submanifold in a Finsler space $(M,F)$. The restriction $F(x,\cdot)|_{T_xM'}$ of the Minkowski norm $F(x,\cdot)$ at each $x\in M'$ defines a Finsler metric $F|_{M'}$
on $M'$. We call $(M',F|_{M'})$ a {\it Finsler submanifold} in $(M,F)$.

We call the submanifold $M'$ or $(M',F)$ {\it totally geodesic} in a $(M,F)$ if for any $x\in M'$ and any nonzero vector $y\in T_xM'$,
the constant speed geodesic $c(t)$ on $(M,F)$ with $c(0)=x$ and $\dot{c}(0)=y$ is contained in $M'$ \cite{XD2017-1}. For example, each connected component of the common fixed point set for a family of isometries is
a closed totally geodesic submanifold \cite{XZ2017}.
Notice that the common zero point set for a family of Killing vector fields can be viewed as the common fixed point set for the
isometries they generate.

There are several equivalent descriptions for a totally geodesic submanifold.

The geodesics on $(M',F|_{M'})$ can be similarly defined as on the ambient space,
by a local minimizing principle for the arc length functional. The submanifold $M'$ is totally geodesic in $(M,F)$ iff any geodesic on $(M',F|_{M'})$ is also a geodesic on
$(M,F)$.

In this paper, we only discuss geodesics with positive constant speeds (i.e., {\it constant speed geodesics}), which
can be equivalently described as following. The
non-constant smooth curve $c(t)$ is a constant speed geodesic on
$(M,F) $ iff $(c(t),\dot{c}(t))$ is the integration curve of the {\it geodesic spray} $\mathrm{G}_M$, which is a globally defined smooth tangent vector field on $TM\backslash 0$ (see \cite{BCS2000} for its explicit description by standard local coordinates). For the Finsler submanifold $(M',F|_{M'})$, we have
another geodesic spray $\mathrm{G}_{M'}$ on $TM'\backslash 0$.
By fundamental knowledge on ordinary differential equations,
it is easy to prove the following lemma.

\begin{lemma} \label{lemma-total-geo-spray}
The following statements are equivalent for the submanifold $M'$ in a Finsler space $(M,F)$:
\begin{enumerate}
\item $M'$ is totally geodesic in $(M,F)$;
\item The restriction of $\mathrm{G}_M$ to $TM'\backslash 0$ is a tangent vector field on $TM'\backslash 0$;
\item The restriction of $\mathrm{G}_M$ to $TM'\backslash 0$ coincides with $\mathrm{G}_{M'}$.
\end{enumerate}
\end{lemma}

Totally geodesic submanifolds provide a crucial technique in the
study of the flag curvature because of the following easy lemma
(see Proposition 5 in \cite{XD2017-1}).

\begin{lemma} \label{lemma-flag-curvature-totally-geodesic}
Let $(M',F|_{M'})$ be a totally geodesic Finsler submanifold in the
Finsler space $(M,F)$ with $\dim M'>1$. Then
for any flag triple $(x,y,\mathbf{P})$ of the totally geodesic Finsler submanifold $(M',F|_{M'})$, we have $K^F(x,y,\mathbf{P})=K^{F|_{M'}}(x,y,\mathbf{P})$.
\end{lemma}

It follows Lemma \ref{lemma-flag-curvature-totally-geodesic} that if $(M,F)$ has positive flag curvature, has non-negative flag curvature or satisfies the (FP) condition, so
does $(M',F|_{M'})$ respectively.

\subsection{Some basic knowledge on Lie theory}
\label{sec-lie-theory}

Let $G$ be any connected Lie group with the Lie algebra $\mathfrak{g}=\mathrm{Lie}(G)$. The radical $\mathfrak{rad}(\mathfrak{g})$ and the nilradical
$\mathfrak{nil}(\mathfrak{g})$ are
the maximal solvable ideal and the maximal nilpotent ideal
of $\mathfrak{g}$ respectively. They generate closed connected normal subgroups $\mathrm{Rad}(G)$ and $\mathrm{Nil}(G)$ respectively (see Proposition 16.2.2 in \cite{HN2012}).
By Corollary 5.4.15 in \cite{HN2012}, we have
$[\mathfrak{rad}(\mathfrak{g}),\mathfrak{g}]\subset
\mathfrak{nil}(\mathfrak{g})$.

There exists a semi-simple subalgebra $\mathfrak{s}$ of $\mathfrak{g}$, such that we have a linear decomposition $\mathfrak{g}=\mathfrak{rad}(\mathfrak{g})+\mathfrak{s}$, called
a {\it Levi decomposition} (see Theorem 5.6.6 in \cite{HN2012}). The Levi subalgebra
$\mathfrak{s}$ is unique up to
$\mathrm{Nil}(G)$-conjugations (see Theorem 5.6.13 in \cite{HN2012}).
We further decompose $\mathfrak{s}$ as a Lie algebra direct sum $\mathfrak{s}=\mathfrak{s}_c\oplus\mathfrak{s}_{nc}$, where
$\mathfrak{s}_c$ and $\mathfrak{s}_{nc}$ are the compact and non-compact part of $\mathfrak{s}$ respectively.

We call the connected Lie group $G$ or its Lie algebra $\mathfrak{g}$ {\it real reductive}, if we have a Lie algebra
direct sum $\mathfrak{g}=\mathfrak{c}(\mathfrak{g})\oplus\mathfrak{g}_c
\oplus\mathfrak{g}_{nc}$ in which  $\mathfrak{g}_c$ and $\mathfrak{g}_{nc}$ are the semi-simple ideals of compact and non-compact types respectively. We call $\mathfrak{g}^u$ the {\it compact dual} of the real reductive Lie algebra $\mathfrak{g}$, if $\mathfrak{g}^u=c(\mathfrak{g})\oplus
\mathfrak{g}_c\oplus\mathfrak{g}_{nc}^u$ where $\mathfrak{g}_{nc}^u$ is a compact dual of $\mathfrak{g}_{nc}$.  We denote  $\mathfrak{g}_{nc}=\mathfrak{k}_{nc}+\mathfrak{p}_{nc}$
the {\it Cartan decomposition} for $\mathfrak{g}_{nc}$, i.e., it satisfies
$[\mathfrak{k}_{nc},\mathfrak{p}_{nc}]\subset\mathfrak{p}_{pc}$
 and $[\mathfrak{p}_{nc},\mathfrak{p}_{nc}]\subset\mathfrak{k}_{nc}$,
then
$\mathfrak{g}^u_{nc}=\mathfrak{k}_{nc}+\sqrt{-1}\mathfrak{p}_{nc}$
is the Cartan decomposition for a compact dual $\mathfrak{g}^u_{nc}$ of $\mathfrak{g}_{nc}$ \cite{He1978}.
If we have $\mathfrak{g}_{nc}=0$ for a real reductive $\mathfrak{g}$, we call $\mathfrak{g}$ a {\it compact Lie algebra}
and $G$ a {\it quasi-compact Lie group}.

The {\it rank}
of $\mathfrak{g}$, denoted as $\mathrm{rk}\mathfrak{g}$, is the real dimension of a Cartan subalgebra in $\mathfrak{g}$. In particular,
we have $\mathrm{rk}\mathfrak{g}=\mathrm{rk}\mathfrak{g}^u=\dim\mathfrak{g}_0+
\mathrm{rk}\mathfrak{g}_c+\mathrm{rk}\mathfrak{g}_{nc}$
for a real reductive $\mathfrak{g}$ and its compact dual.

We call a subgroup $H$ {\it compactly imbedded} in
$G$, if its image $\mathrm{Ad}_\mathfrak{g}(H)\subset\mathrm{Inn}\mathfrak{g}=G/C(G)$ has a compact closure. Similarly, we call a subalgebra $\mathfrak{h}$
{\it compactly imbedded} if the connected subgroup $H$ it generates in $G$ is compactly imbedded. A maximal compactly imbedded subgroup $K$ in $G$ is a closed subgroup, and by Theorem 14.1.3 in \cite{HN2012}, it is unique up to $\mathrm{Ad}(G)$-conjugations. So for
any compactly imbedded subgroup $H$ of $G$, we can find
a maximal compactly imbedded subgroup $K$ such that $H\subset K\subset G$.

\subsection{Homogeneous Finsler space}

A connected Finsler space $(M,F)$ is called {\it homogeneous} if some connected Lie group $G$ acts isometrically and transitively on $(M,F)$ \cite{De2012}. In later discussion, we usually present a homogeneous Finsler space as $(G/H,F)$.
Here $H$ is the isotropy subgroup at $o=eH\in G/H$. We denote
$\mathfrak{g}$ and $\mathfrak{h}$ the Lie algebras of $G$ and
$H$ respectively.

If the $G$-action on $G/H$ is effective, i.e., $G$ is a subgroup of the connected isometry group $I_0(G/H,F)$, $H$ is compactly imbedded in $G$. Furthermore, if $G$ is a closed subgroup of $I_0(G/H,F)$, $H$ is compact itself.

These basic observations imply that we can always find an
$\mathrm{Ad}(H)$-invariant linear direct sum $\mathfrak{g}=\mathfrak{h}+\mathfrak{m}$, which is called
a {\it reductive decomposition}. For example, when the $G$-action
on $(G/H,F)$ is effective, the restriction of the Killing form
$B_\mathfrak{g}$ of $\mathfrak{g}$ to $\mathfrak{h}$ is non-degenerate, so the $B_\mathfrak{g}$-orthogonality provides a special reductive decomposition $\mathfrak{g}=\mathfrak{h}+\mathfrak{m}$ satisfying
$\mathfrak{nil}(\mathfrak{g})\subset\mathfrak{m}$ (see Lemma 2 and Remark 1 in \cite{Ni2017}).

Given a reductive decomposition $\mathfrak{g}=\mathfrak{h}+\mathfrak{m}$
for the homogeneous Finsler space
$(G/H,F)$, the subspace $\mathfrak{m}$ can be identified with
the tangent space at $o=eH$, such that the $\mathrm{Ad}(H)$-action on $\mathfrak{m}$ coincides with the
isotropy $H$-action on $T_{o}(G/H)$. The $G$-invariant Finsler metric $F$ is one-to-one determined by its restriction to $T_{o}(G/H)$, i.e., an $\mathrm{Ad}(H)$-invariant Minkowski norm
on $\mathfrak{m}$ \cite{De2012}.

For any $u\in\mathfrak{g}$, we denote $u=u_\mathfrak{h}+u_\mathfrak{m}$ with $u_\mathfrak{h}\in\mathfrak{h}$ and
$u_\mathfrak{m}\in\mathfrak{m}$ the decomposition of $u$,
and $\mathrm{pr}_\mathfrak{h}(u)=u_\mathfrak{h}$
and $\mathrm{pr}_\mathfrak{m}(X)=u_\mathfrak{m}$
the corresponding linear projections, for a given
reductive decomposition $\mathfrak{g}=\mathfrak{h}+\mathfrak{m}$.

The flag curvature of homogeneous Finsler space is an important topic which has been extensively studied in recent years \cite{DX-survey,Huang17}.
The general homogeneous flag curvature formula was given in
\cite{Huang2015}.
However, for a globally symmetric Finsler space $(G/H,F)$ \cite{De2012}, the flag curvature formula can be much simplified
as following.

\begin{lemma}\label{prop-symmetric-space-non-compact}
Assume that the homogeneous Finsler space $(G/H,F)$ admits
a Cartan decomposition, i.e., a reductive decomposition
$\mathfrak{g}=\mathfrak{h}+\mathfrak{m}$
satisfying
$[\mathfrak{m},\mathfrak{m}]\subset\mathfrak{h}$, then for
any linearly independent pair $u,v\in\mathfrak{m}=T_o(G/H)$
and the tangent plane they span, we have
$$K^F(o,u,\mathbb{R})=\frac{\langle [[v,u],v],u\rangle_u^F}{
\langle u,u\rangle_u^F\langle v,v\rangle_u^F-
\langle u,v\rangle_u^F\langle u,v\rangle_u^F}.$$
\end{lemma}

In particular, we see immediately that any globally
symmetric Finsler space $(G/H,F)$ which rank is bigger than 1 does not satisfy the (FP) condition, because we can find a commutative subalgebra
$\mathbf{P}\subset\mathfrak{m}=T_o(G/H)$ with respect to a Cartan decomposition for $G/H$, such that $K^F(o,u,\mathbf{P})=0$ for any nonzero vector $u\in\mathbf{P}$.

When $(G/H,F)$
is a globally symmetric Finsler space of non-compact type, i.e., $\mathfrak{g}$ is a
semi-simple Lie algebra of non-compact type and $\mathfrak{h}$ is maximal compactly imbedded in
$\mathfrak{g}$, it is not hard to prove by Theorem 5.8 in \cite{De2012} or Lemma \ref{prop-symmetric-space-non-compact}
that $(G/H,F)$ has non-positive flag curvature and negative Ricci scalar. So in this case $(G/H,F)$
 can not be non-negatively curved.
\section{Some techniques for studying g.o.~Finsler space}

\subsection{Homogeneous geodesic and g.o.~Finsler space}

Let $(M,F)$ be a connected Finsler space on which a connected Lie group
$G$ acts isometrically. We call a non-constant geodesic $c(t)$ {\it $G$-homogeneous}, if $c(t)=\exp tw\cdot x$ for some $w\in\mathfrak{g}=\mathrm{Lie}(G)$ and $x\in M$. It is obvious to see that
any $G$-homogeneous geodesic $c(t)$ has a positive constant speed.
We call $(M,F)$ a {\it $G$-geodesic orbit Finsler space} ({\it$G$-g.o.~Finsler space} in short) if each constant speed geodesic on $(M,F)$ is $G$-homogeneous \cite{YD2014}. If the group $G$ has not been specified, the connected isometry group $I_0(M,F)$ is automatically chosen.

Since any connected $G$-g.o.~Finsler space $(M,F)$ is a homogeneous Finsler space on which $G$ acts isometrically and transitively, we may present it as $(M,F)=(G/H,F)$, in which $H$ is the isotropy subgroup at $o=eH\in G/H$. We call $X\in\mathfrak{g}$ a {\it $G$-geodesic vector} if $\exp tX\cdot o$ is a non-constant geodesic. When $G\subset I_0(M,F)$, i.e., $G$ acts effectively on $M$, we can find a
reductive decomposition $\mathfrak{g}=\mathfrak{h}+\mathfrak{m}$.
Reformulating Proposition 3.1 in \cite{Xu2018}, we
have the following lemma.

\begin{lemma}\label{lemma-go-equi-description}
Let $(G/H,F)$ be a homogeneous Finsler space with a reductive decomposition $\mathfrak{g}=\mathfrak{h}+\mathfrak{m}$. Then
for any $G$-geodesic vector $w\in\mathfrak{g}$, $\mathrm{pr}_\mathfrak{m}(w)$ is a nonzero vector in $\mathfrak{m}$, and
$(G/H,F)$ is a $G$-g.o.~Finsler space iff the linear projection $\mathrm{pr}_\mathfrak{m}$ maps
the subset of all $G$-geodesic vectors onto $\mathfrak{m}\backslash \{0\}$.
Further more, the following statements are equivalent for a nonzero vector $u\in\mathfrak{m}$:
\begin{enumerate}
\item There exists some $G$-geodesic vector $w$, such that $\mathrm{pr}_\mathfrak{m}(w)=u$;
\item There exists some $v\in\mathfrak{h}$ such that
$\langle[u+v,\mathfrak{m}]_\mathfrak{m},u
\rangle^F_{u}=0$;
\item The spray vector field $\eta(\cdot)$ is tangent to the $\mathrm{Ad}(H)$-orbit at $u$.
\end{enumerate}
\end{lemma}

Here the spray vector field $\eta:\mathfrak{m}\backslash\{0\}\rightarrow\mathfrak{m}$
for a homogeneous Finsler space $G/H$ with a reductive decomposition $\mathfrak{g}=\mathfrak{h}+\mathfrak{m}$ is
defined by
\begin{equation}\label{0001}
\langle\eta(y),u\rangle_y^F=
\langle y,[u,y]_\mathfrak{m}\rangle_y^F,\quad\forall u\in\mathfrak{m},
\end{equation}
where $\langle\cdot,\cdot\rangle_y^F$ is the fundamental tensor
for the $\mathrm{Ad}(H)$-invariant Minkowski norm $F(o,\cdot)$ on $\mathfrak{m}$
\cite{Huang2015}.

Notice that the same Finsler space $(M,F)$ may have g.o.~ properties with respect to different isometric group actions.
For example, the standard Riemannian metric on a unit sphere $S^n(1)$
is $G$-g.o.~for any connected subgroup $G\subset SO(n+1)$ which acts transitively on $S^n(1)$ \cite{BN2014,Ni2013}.

A Finsler space $(M,F)$ is $G$-g.o.~iff it is $G'$-g.o.~in which
$G'$ is the image of $G$ in $I(M,F)$. Generally speaking, the smaller is $\mathfrak{g}'=\mathrm{Lie}(G')$, the stronger is
the corresponding $G$-g.o.~property.

In a previous work, we have proved the following lemma (see Lemma 4.2 in \cite{Xu2018}).

\begin{lemma}\label{lemma-go-commuting-isometries}
Let $(G/H, F)$ be a connected G-geodesic orbit space, and $X$ a smooth vector field on $G/H$ commuting
with all the Killing vector fields in $\mathfrak{g} = \mathrm{Lie}(G)$, then $X$ is a Killing vector field on $(G/H, F)$.
\end{lemma}

Lemma \ref{lemma-go-commuting-isometries} can be applied to
discuss the following simple example.
\begin{lemma}\label{lemma-left-invariant-go}
Assume that the connected Lie group $G$ admits a left invariant Finsler metric $F$ which is g.o.~with respect to $G=L(G)$ for all the left multiplications, then $F$ is bi-invariant
and $G$ has a compact Lie algebra.
\end{lemma}

\begin{proof}
Killing vector fields on $(G,F)$ in the Lie algebra of $L(G)$ are right invariant vector fields on $G$. Let $X$ be any left invariant vector field on $G$. Since $X$
commutes with all right invariant vector fields on $G$, by Lemma
\ref{lemma-go-commuting-isometries},  $X$ is also a Killing vector field on $(G,F)$. So $F$ is
bi-invariant, and then $F(e,\cdot)$ is an $\mathrm{Ad}(G)$-invariant Minkowski norm on $\mathfrak{g}=\mathrm{Lie}(G)$. By an averaging process, we can use $F(e,\cdot)$ to construct an
$\mathrm{Ad}(G)$-invariant inner product on $\mathfrak{g}$ \cite{XD2017-1}. So
$\mathfrak{g}$ is compact, which
ends the proof of this lemma.
\end{proof}
\subsection{Totally geodesic techniques}
\label{sec-total-geo}
We have mentioned that fixed point sets of isometries provide totally geodesic submanifolds, so they are crucial for studying the flag curvature in Finsler geometry.
This technique is also important for studying the g.o.~manifolds
because of the following lemma.

\begin{lemma}\label{lemma-go-fix-point-tech}
Let $G$ be a connected Lie group, $(G/H,F)$ be a connected $G$-g.o.~Finsler space, and $L$ a subset in $H$. Then $\mathrm{Fix}_o(L,G/H)$, the connected component of the fixed point set of $L$ in $G/H$ containing $o=eH$, is a totally geodesic submanifold in $(G/H,F)$, and
the submanifold metric $F|_{\mathrm{Fix}_o(L,G/H)}$ is a $G'$-g.o.~ Finsler metric, in which $G'$ is the identity component of the centralizer $C_G(L)$ of $L$ in $G$.
\end{lemma}

Notice that if $L$ is a closed connected subgroup of $H$, the
group $G'$ in Lemma \ref{lemma-go-fix-point-tech} can be replaced
by the quotient Lie group $G'/L$.

Lemma \ref{lemma-go-fix-point-tech} is a reformulation of
Lemma 3.4
in \cite{XDY-preprint}.
To be more self contained, we briefly sketch its proof here.

It is obvious that $\mathrm{Fix}_o(L,G/H)$ is totally geodesic in $(G/H,F)$, with an isometric $G'$-action.

To prove the $G'$-g.o.~property, we may simply
assume $L$ is a closed subgroup in $H$.
For any geodesic of positive constant speed $c(t)$ on
$\mathrm{Fix}_o(L,G/H)$, we may present it as $c(t)=\exp(tX)\cdot x$
for some $X\in\mathfrak{g}$. We can average
$\mathrm{Ad}(g)X$ for all $g\in L$, to get a vector $\overline{X}\in\mathfrak{g}'=\mathrm{Lie}(G')=
\mathfrak{c}_\mathfrak{g}(L)$ which still satisfies
$c(t)=\exp(t\overline{X})\cdot x$. Then the claim of the $G'$-g.o.
property is proved.

To make the averaging argument precise, we need to switch to the closure
$\overline{L}$ of $L$ in $I_0(G/H,F)$. Then the $\mathrm{Ad}(\overline{L})$-actions preserves ${\mathfrak{g}}$. Further more
$\overline{L}$ is compact, so
we can find a bi-invariant volume form $d\mathrm{Vol}$ on $\overline{L}$
with
$\int_{\overline{L}}d\mathrm{Vol}=1$.
Then for any $X\in\mathfrak{g}$, the average
$\overline{X}=\int_{g\in\overline{L}}\mathrm{Ad}(g)X d\mathrm{Vol}_{\overline{L}}$
meets our requirement.

In \cite{XDY-preprint}, we have also proved the following totally geodesic technique.

\begin{lemma}\label{lemma-go-total-geod-tech}
Let $(G/H,F)$ be a connected $G$-g.o.~Finsler space in which both $G$ and $H$ are connected, and $G'$
a connected subgroup of $G$ containing  $H$, then $G'/H$ is a
totally geodesic submanifold in $(G/H,F)$, and a $G'$-g.o.~Finsler
space itself.
\end{lemma}

Lemma \ref{lemma-go-total-geod-tech} is a generalization of Proposition 1.9 in \cite{Go1996}. The connectedness assumptions
for $H$ and $G'$ in this lemma is only for the convenience in later discussion, and not necessary.

Let us briefly recall its proof, which can be easily reduced to the situation that the $G$-action on $G/H$ is effective. Then we fix a $B_\mathfrak{g}$-orthogonal
$\mathrm{Ad}(H)$-invariant linear decomposition $\mathfrak{g}=\mathfrak{h}+\mathfrak{m}'+\mathfrak{m}''$, such that
$\mathfrak{g}'=\mathrm{Lie}(G')=\mathfrak{h}+\mathfrak{m}'$.

The geodesic spray $\mathrm{G}(x,y)$ for $(G/H,F)$
can be determined by the spray vector field $\eta(\cdot)$
with respect to the reductive decomposition
$\mathfrak{g}=\mathfrak{h}+\mathfrak{m}=\mathfrak{h}+
(\mathfrak{m}'+\mathfrak{m}'')$ for $(G/H,F)$.
Since $(G/H,F)$ is $G$-g.o., by Lemma \ref{lemma-go-equi-description}, $\eta(y)$ is tangent to
$\mathrm{Ad}(H)y$ for any nonzero vector
$y\in\mathfrak{m}'$. So we have
$\eta(\mathfrak{m}'\backslash\{0\})\subset\mathfrak{m}'$.
It is easy to check that in this situation $\eta|_{\mathfrak{m}'\backslash\{0\}}$ coincides with the spray
vector field of $(G'/H,F|_{G'/H})$, with respect to the reductive
decomposition $\mathfrak{g}'=\mathfrak{h}+\mathfrak{m}'$.
Applying Lemma \ref{lemma-go-equi-description}, we see $(G'/H,F')$
is $G'$-g.o..
Meanwhile, because $\eta(\mathfrak{m}'\backslash\{0\})\subset\mathfrak{m}'$, we see that the geodesic spray $\mathrm{G}(x,y)$ of $(G/H,F)$ is tangent to $T(G'/H)\backslash 0$ when $x=o$. By homogeneity and Lemma \ref{lemma-total-geo-spray}, $(G'/H,F')$
is totally geodesic in $(G/H,F)$.

\subsection{Submersion techniques}
\label{sec-sub-tech}
In \cite{XDY-preprint}, we have proved the following submersion technique
for g.o.~Finsler space.

\begin{lemma}\label{lemma-go-submersion-1}
Let $(G/H,F)$ be a connected $G$-g.o.~Finsler space where $G$ is
a closed subgroup of $I_0(G/H,F)$. Then for any closed normal subgroup $H'$ of $G$, there exists a unique $G$-g.o.~Finsler
metric $F'$ on $G/HH'$, defined by submersion from the natural projection from $G/H$ to $G/HH'$ and the metric $F$ on $G/H$.
\end{lemma}

The requirement for $G$ in Lemma \ref{lemma-go-submersion-1} guaranteed that $H$ is compact.
So $H'H$ is a closed subgroup of $G$, and then $G/H'H$ is a
smooth coset space.
We identify $G/HH'$ with the
orbit space $H'\backslash G/H$ for the isometries of all
left $H'$-actions on $G/H$, then Lemma \ref{lemma-0}
provides a metric $F'$ on $H'\backslash G/H=G/H'H$. Proving this $F'$ meets
the requirements in the lemma is easy, so we omit it here. See
Theorem 3.5 and its proof in \cite{XDY-preprint}.

In later discussion, we will also use the following submersion technique.

\begin{lemma}\label{lemma-go-submersion-2}
Let $(G/H,F)$ be a connected $G$-g.o.~Finsler space in which both $G$ and $H$ are connected. Suppose that $H'$ is a closed connected
subgroup of $G$ such that $H$ is normal in $H'$. Then there exists
a $G$-g.o.~Finsler metric $F'$ on $G/H'$, such that the natural projection from $(G/H,F)$ to $(G/H',F')$ is a submersion.
\end{lemma}

\begin{proof} The
left $G$-multiplications and right $H'$-multiplications on $G$
induce two connected Lie groups of diffeomorphisms on $G/H$,
i.e., left $G$-actions and right $H'$-actions, which commute with each other. Since $(G/H,F)$ is $G$-g.o., by Lemma \ref{lemma-go-commuting-isometries}, right $H'$-actions on $G/H$ are also isometries. The orbit space of right $H'$-actions on $G/H$
is the smooth coset space $G/H'$, and the quotient map coincides
with the natural projection $\pi:G/H\rightarrow G/H'$. By Lemma \ref{lemma-0}, there exists a Finsler metric $F'$ on $G/H'$ induced by submersion and the metric $F$. By the uniqueness of this $F'$, left $G$-actions on $G/H$ induce isometries on $(G/H',F')$, i.e., $F'$ is $G$-invariant.

To prove $F'$ is $G$-g.o., we consider any constant speed geodesic $c_2(t)$ on $(G/H',F')$. Denote $x'=c_2(0)$ and $x$ any point in
$\pi^{-1}(x')$, then there exists a unique horizonal lifting $c_1(t)$ of $c_2(t)$ such that $c_1(t)$ is a constant speed geodesic with $c_2(0)=x$. Since $(G/H,F)$ is $G$-g.o., we can find
a vector $u\in\mathfrak{g}$, such that $c_1(t)=\exp tu\cdot x$.
Then we also have $c_2(t)=\exp tu\cdot x'$ for all $t\in\mathbb{R}$.
So $(G/H',F')$ is $G$-g.o., which ends the proof of this lemma.
\end{proof}

\section{A special class of non-negatively curved g.o.~Finsler spaces}
\label{sec-a-special-class-non-neg}

%
%
%
%
%
%

As the preparation for proving Theorem \ref{main-thm},
we will use the totally geodesic and submersion techniques in
Section \ref{sec-total-geo} and Section \ref{sec-sub-tech}
to prove the following theorem.

\begin{theorem}\label{thm-preparation}
Assume $(G/H,F)$ is a connected $G$-g.o.~Finsler
space satisfying the following conditions:
\begin{enumerate}
\item $\mathfrak{g}=\mathrm{Lie}(G)$ is a real reductive Lie algebra, $\mathfrak{h}=\mathrm{Lie}(H)$ is compact, and $\mathrm{rk}\mathfrak{g}\leq\mathrm{rk}\mathfrak{h}+1$;
\item $(G/H,F)$ has non-negative flag curvature.
\end{enumerate}
Then $\mathfrak{g}$ is a compact Lie algebra.
\end{theorem}
\begin{proof}
For simplicity,
we may assume $G$ and $H$ are connected, i.e.,
replace $G$ and $H$ with their identity components $G_0$ and $H_0$ respectively, because $F$ induces
a locally isometric $G$-g.o.~Finsler metric on $G_0/H_0$, meeting all requirements in the theorem.

We denote
$\mathfrak{g}=\mathfrak{c}(\mathfrak{g})
\oplus\mathfrak{g}_c\oplus
\mathfrak{g}_{nc}$ the decomposition for the real reductive $\mathfrak{g}$.
We can find a maximal compactly imbedded subalgebra $\mathfrak{k}=\mathfrak{c}(\mathfrak{g})\oplus\mathfrak{g}_c
\oplus\mathfrak{k}_{nc}$ such that $\mathfrak{h}\subset\mathfrak{k}$. Denote
$\mathfrak{g}=\mathfrak{k}+\mathfrak{p}$ and
$\mathfrak{g}^u=\mathfrak{k}+\sqrt{-1}\mathfrak{p}$
the Cartan decompositions for $\mathfrak{g}$ and $\mathfrak{g}^u$
respectively, where $\mathfrak{p}$ is
the $B_\mathfrak{g}$-orthogonal complement of $\mathfrak{k}_{nc}$ in $\mathfrak{g}_{nc}$.
Let $\mathfrak{m}_1$ be the $B_\mathfrak{g}$-orthogonal complement of $\mathfrak{h}$ in $\mathfrak{k}$, then we have a $B_\mathfrak{g}$-orthogonal $\mathrm{Ad}(H)$-invariant linear decomposition
$$\mathfrak{g}=\mathfrak{h}+\mathfrak{m}_1+\mathfrak{p}.$$

We choose a Cartan subalgebra
$\mathfrak{t}_\mathfrak{h}$ of $\mathfrak{h}$, extend it
to a Cartan subalgebra $\mathfrak{t}_\mathfrak{k}$ of $\mathfrak{k}$, and extend it again to a Cartan subalgebra $\mathfrak{t}$ of
$\mathfrak{g}$. This $\mathfrak{t}$ satisfies
$$\mathfrak{t}=
\mathfrak{t}\cap\mathfrak{h}+\mathfrak{t}
\cap\mathfrak{m}_1+\mathfrak{t}\cap\mathfrak{p}.$$
Correspondingly, we have the Cartan subalgebra $\mathfrak{t}^u=\mathfrak{t}\cap\mathfrak{k}+
\sqrt{-1}\mathfrak{t}\cap\mathfrak{p}$
of $\mathfrak{g}^u$, with respect to which, we have the root plane
decomposition
$$\mathfrak{g}^u=\mathfrak{t}^u+\sum_{\alpha\in\Delta}
\mathfrak{g}^u_{\pm\alpha},$$
where $\Delta\subset(\mathfrak{t}^u)^*$ is the root system and $\mathfrak{g}^u_{\pm\alpha}$ are root planes.

Now we prove Theorem \ref{thm-preparation} by an induction for $\mathrm{rk}\mathfrak{h}$.

Firstly, we prove the lemma when $\mathrm{rk}\mathfrak{h}\leq 1$. When $\mathrm{rk}\mathfrak{h}=0$, i.e., $H=\{e\}$,
$F$ is a $G$-g.o.~left invariant Finsler metric on $G$. By Lemma
\ref{lemma-left-invariant-go},
$\mathfrak{g}$ must be compact.

When $\mathrm{rk}\mathfrak{h}=1$,
by the rank inequality in (1) of Theorem \ref{thm-preparation},
we have
$\mathrm{rk}\mathfrak{g}\leq2$. All the
possible non-compact $\mathfrak{g}$'s can be sorted as following.

{\bf Case 1:} $\mathrm{rk}{\mathfrak{h}}=\mathrm{rk}\mathfrak{k}=
\mathrm{rk}\mathfrak{g}=1$.

In this case, $(G/H,F)$ is a hyperbolic plane, i.e., $G/H=SL(2,\mathbb{R})/SO(2)$ and $F$ is a Riemannian metric with negative constant curvature. This is a contradiction to (2) of
Theorem \ref{thm-preparation}.

{\bf Case 2:} $\mathrm{rk}\mathfrak{k}
<\mathrm{rk}\mathfrak{g}=2$ and $\mathfrak{h}=\mathfrak{k}$.

In this case, $(G/H,F)$ is a globally symmetric Finsler space of non-compact type. It has negative Ricci scalar, contradicting
(2) of Theorem \ref{thm-preparation}.

{\bf Case 3:} $\mathrm{rk}\mathfrak{k}<
\mathrm{rk}\mathfrak{g}$ and $\mathfrak{h}\neq\mathfrak{k}$.

In this case, there are two possibilities, i.e., either
$$\mathfrak{g}=sl(3,\mathbb{R}),  \mathfrak{k}=so(3)  \mbox{ and } \mathfrak{h}=so(2),$$
where $\mathfrak{h}$ corresponds to a $2\times2$-block, or
$$\mathfrak{g}=sl(2,\mathbb{C}), \mathfrak{k}=su(2)\mbox{ and }
\mathfrak{h}=so(2).$$
For both possibilities,
there exists a subalgebra $\mathfrak{g}'=sl(2,\mathbb{R})$ of $\mathfrak{g}$ which contains $\mathfrak{h}$.
Denote $G'$ the connected subgroup of $G$ generated by $\mathfrak{g}'$. By Lemma \ref{lemma-flag-curvature-totally-geodesic}
and Lemma \ref{lemma-go-total-geod-tech}, $(G'/H,F|_{G'/H})$
is totally geodesic in $(G/H,F)$, and non-negatively curved itself. On the other hand, $(G'/H,F|_{G'/H})$ is the
hyperbolic plane with negative constant curvature. This is
a contradiction.

{\bf Case 4:}
$\mathrm{rk}\mathfrak{k}
=\mathrm{rk}\mathfrak{g}=2$, $\mathfrak{g}$ is not simple and $\mathfrak{h}=\mathbb{R}$.

In this case, the non-compact Lie algebra $\mathfrak{g}$ must have an ideal $\mathfrak{g}_1=sl(2,\mathbb{R})$.
Then the subalgebra $\mathfrak{g}'=\mathfrak{h}+\mathfrak{g}_1$
 is either
$sl(2,\mathbb{R})\oplus\mathbb{R}$ or $sl(2,\mathbb{R})$. Denote $G'$ the
connected subgroup generated by $\mathfrak{g}'$.
By Lemma \ref{lemma-flag-curvature-totally-geodesic} and Lemma \ref{lemma-go-total-geod-tech},
$(G'/H,F|_{G'/H})$ is a totally geodesic submanifold in $(G/H,F)$,
and a non-negatively curve $G'$-g.o.~Finsler space itself.
Let $\mathfrak{t}'$ be a Cartan subalgebra of $\mathfrak{g}'$ containing $\mathfrak{h}$ and $T'$ the corresponding connected subgroup generated by $\mathfrak{t}'$. By
Lemma \ref{coro-flag-curv-submersion}
 and
Lemma \ref{lemma-go-submersion-2},
there exists a non-negatively curved homogeneous Finsler metric $F'$ on $G'/T'$ defined by submersion. On the other hand, $(G'/T',F')$ is the hyperbolic plane with negative constant curvature. This is a contradiction.

{\bf Case 5:} $\mathrm{rk}\mathfrak{k}
=\mathrm{rk}\mathfrak{g}=2$, $\mathfrak{g}$ is not simple and
$\mathfrak{h}=su(2)$.

In this case, $\mathfrak{g}=sl(2,\mathbb{R})\oplus su(2)$ and
$\mathfrak{h}$ is the $su(2)$-factor in $\mathfrak{g}$.
Then $F$ is locally isometric to
a left invariant Finsler metric on the Lie group $SL(2,\mathbb{R})$, which is g.o.~with respect to all left multiplications.
This is a contradiction to Lemma \ref{lemma-left-invariant-go}  because $sl(2,\mathbb{R})$ is not compact.

{\bf Case 6:} $\mathrm{rk}\mathfrak{k}
=\mathrm{rk}\mathfrak{g}=2$, $\mathfrak{g}$ is simple, and $\mathfrak{h}=\mathbb{R}$.

In this case, we can find a subalgebra $\mathfrak{g}'=\mathfrak{t}+\sqrt{-1}\mathfrak{g}^u_{\pm\alpha}
=sl(2,\mathbb{R})\oplus\mathbb{R}\subset\mathfrak{g}$,
in which $\mathfrak{g}^u_{\pm\alpha}$ is a root plane of $\mathfrak{g}^u$ contained in
$\sqrt{-1}\mathfrak{p}$.
Let $G'$ be connected subgroup of $G$ generated by $\mathfrak{g}'$, then $G'$ contains $H$, and by
Lemma \ref{lemma-flag-curvature-totally-geodesic} and
Lemma \ref{lemma-go-total-geod-tech}, $(G'/H,F|_{G'/H})$ is a
totally geodesic submanifold in $(G/H,F)$ and a non-negatively curved $G'$-g.o.~Finsler space itself.
We have discussed it in Case 4 and find a contradiction.

{\bf Case 7:} $\mathrm{rk}\mathfrak{k}
=\mathrm{rk}\mathfrak{g}=2$, $\mathfrak{g}$ is simple, and $\mathfrak{h}=su(2)$ is an ideal of $\mathfrak{k}$.

By Lemma \ref{coro-flag-curv-submersion}
and
Lemma \ref{lemma-go-submersion-2}, there exists a non-negatively curved
homogeneous Finsler metric $F'$ on $G/K$
defined by submersion. But $(G/K,F')$
is a globally symmetric Finsler space of non-compact type. This is a contradiction because $(G/K,F')$ has negative Ricci scalar.

{\bf Case 8:} $\mathrm{rk}\mathfrak{k}
=\mathrm{rk}\mathfrak{g}=2$, $\mathfrak{g}$ is simple and $\mathfrak{h}=su(2)$ is not an ideal of $\mathfrak{k}$.

In this final case, only two possibilities remains.

{\bf Subcase 8.1} $\mathfrak{g}=sp(2)$, $\mathfrak{k}=sp(1)\oplus sp(1)$ and $\mathfrak{h}= sp(1)$ diagonally imbedded in $\mathfrak{k}$.

Let $T_H$ be the maximal torus in $H$ generated by $\mathfrak{t}_\mathfrak{h}=\mathfrak{t}\cap\mathfrak{h}$.
The restriction of $F$ to the fixed point set $\mathrm{Fix}_o(T_H,G/H)$
is locally isometric to a left invariant Finsler
metric on the Lie group $SL(2,\mathbb{R})$, which is g.o.~with
respect to all left multiplications. This is a contradiction to Lemma \ref{lemma-left-invariant-go} because $sl(2,\mathbb{R})$ is not compact.

{\bf Subcase 8.2:} $\mathfrak{g}=G_2$, $\mathfrak{k}=su(2)\oplus su(2)$ and $\mathfrak{h}=su(2)$ diagonally imbedded in $\mathfrak{k}$.

We can find a long root $\alpha$ and a short root $\beta$ for $\mathfrak{k}$, such that
$\mathfrak{t}_\mathfrak{h}=\mathrm{ker}(\alpha-\beta)$.
Notice that $\frac12\alpha+\frac32\beta$ and $\frac12\alpha-\frac12\beta$ are an orthogonal pair of roots in $\Delta$.
We can find an element $g\in T_H=\exp\mathfrak{t}_\mathfrak{h}$, such that $\mathrm{Ad}(g)$ coincides with the identity map $\mathrm{Id}$ when restricted to
$\mathfrak{g}^u_{\pm(\frac12\alpha+\frac32\beta))}$ and
$\mathfrak{g}^u_{\pm(\frac12\alpha-\frac12\beta)}$, and
and $-\mathrm{Id}$ when restricted to any other root plane in
$\mathfrak{g}^u$.
By Lemma \ref{lemma-go-fix-point-tech}, the fixed point set $(\mathrm{Fix}_o(g,G/H),F|_{\mathrm{Fix}_o(g,G/H)})$
is a non-negatively curved
totally geodesic submanifold in $(G/H,F)$, and it is
locally isometric to an $SL(2,\mathbb{R})\times SL(2,\mathbb{R})$-g.o.~metric on $(SL(2,\mathbb{R})\times SL(2,\mathbb{R}))/(SO(2)\times\{I\})$. This is a situation which
has been discussed in Case 4 where we have found a contradiction.

To summarize, we have proved by contradiction that $\mathfrak{g}$ must be compact, i.e., the theorem is valid, when $\mathrm{rk}\mathfrak{h}=1$.

Nextly, we assume that the theorem is valid when $\mathrm{rk}\mathfrak{h}=n-1$ with $n>1$
and then prove the case
$\mathrm{rk}\mathfrak{h}=n$ by contradiction.

We assume conversely that $\mathfrak{g}$ is not compact
when $\mathrm{rk}\mathfrak{h}=n$.
Let $\alpha\in\Delta\subset(\mathfrak{t}^u)^*$ be a root of $\mathfrak{g}^u$ such that
the root plane $\mathfrak{g}^u_{\pm\alpha}$ is not
contained in $\mathfrak{k}$. The codimension of
$\ker \alpha \cap \mathfrak{t}_\mathfrak{h}$ in $\mathfrak{t}$ is $1$ or $2$, when
$\mathrm{rk}\mathfrak{g}=n$ or $n+1$ respectively. So $\ker\alpha\cap\mathfrak{t}_\mathfrak{h}$ always has a positive dimension.

Let $T'$ be any one-dimension sub-torus in $T$ generated by a line
$\mathfrak{t}'\subset\ker\alpha\cap\mathfrak{t}_\mathfrak{h}$,
and $G'=(C_G(T'))_0/T'$ the identity component of $C_G(T')/T'$. The Lie algebra
$\mathfrak{g}'=\mathrm{Lie}(G')$ is a real reductive Lie algebra
with $\mathrm{rk}\mathfrak{g}'=\mathrm{rk}\mathfrak{g}-1$.

By Lemma \ref{lemma-flag-curvature-totally-geodesic} and Lemma \ref{lemma-go-fix-point-tech},
the fixed point set $(\mathrm{Fix}_o(T',G/H),F|_{\mathrm{Fix}_o(T',G/H)})$ is a totally geodesic submanifold in $(G/H,F)$, and a non-negatively curved $G'$-g.o.
Finsler space itself. If we present $\mathrm{Fix}_o(T',G/H)$ as
$G'/H'$, $H'=((C_G(T'))_0\cap H)/T'$ is compact because $H$ is
compact and $C_G(T')$ is closed in $G$.
The rank of $\mathfrak{h}'=\mathrm{Lie}(H')=\mathfrak{c}_{\mathfrak{h}(
\mathfrak{t}')}/\mathfrak{t}'$ is $n-1$. So the rank inequality
$\mathrm{rk}\mathfrak{g}'\leq\mathrm{rk}\mathfrak{h}'+1$ is satisfied.
The inductive assumption claim $\mathfrak{g}'$ is compact.
But this is not true because of
the root plane
$\mathfrak{g}^u_{\pm\alpha}$.

To summarize, we see the theorem is valid when $\mathrm{rk}\mathfrak{h}=n$, which ends the proof by induction.
\end{proof}

At the end of this section, we remark that Theorem \ref{thm-preparation} answers the following question positively
in a special case.

\begin{question}\label{Question-1}
Let $(M,F)$ be a connected simply connected g.o.~Finsler manifold
which has non-negative flag curvature. Then can we find
a quasi-compact Lie group acting isometrically and transitively on
$(M,F)$?
\end{question}

More generally, we may ask

\begin{question}\label{Question-2}
Let $(M,F)$ be a connected simply connected homogeneous Finsler manifold which has non-negative flag curvature. Then can we find
a quasi-compact Lie group acting isometrically and transitively on
$(M,F)$?
\end{question}

Until now, these questions are still open except in some special cases. For example, we can answer Question \ref{Question-1} positively when $(M,F)$ is normal homogeneous or $\delta$-homogeneous, by Theorem 3.6 in
\cite{XZ2018} or the Finsler version of Corollary 4.6 in \cite{XN2019}.
We can answer Question \ref{Question-2} positively when $(M,F)$
is Riemannian by the splitting theorem, or when $(M,F)$ is positively curved by Bonnet-Myers theorem.
\section{Proof of Theorem \ref{main-thm}}

In this section, we consider a homogeneous Finsler space $(G/H,F)$
with $\dim G/H>1$ satisfying the following requirements:
\begin{enumerate}
\item $G$ is a closed connected subgroup of $I_0(G/H,F)$;
\item $(G/H,F)$ is $G$-g.o.;
\item $(G/H,F)$ has non-negatively flag curvature and satisfies
the (FP) conditions.
\end{enumerate}

The condition (1) is equivalent to the following,
$$G \mbox{ acts effectively on } G/H\mbox{ and the isotropy subgroup }H\mbox{ is compact.}$$
So if we change $H$ with
its identity component $H_0$ and $F$ with the locally isometric $\tilde{F}$ that $F$ induces on $G/H_0$, (1) is still valid.
The conditions (2) and (3) are not affected by this change because they are only relevant to the local geometry and the Lie algebra. To summarize, for simplicity of later discussion,
we may further assume $H$ is connected.

We fix the $B_\mathfrak{g}$-orthogonal reductive decomposition $\mathfrak{g}=\mathfrak{h}+\mathfrak{m}$. Then the nilradical
$\mathfrak{n}=\mathfrak{nil}(\mathfrak{g})$ is contained in
$\mathfrak{m}$.

\subsection{Real reductiveness of $\mathfrak{g}$}

\begin{lemma}\label{step-1}
Keep all assumptions and notations of this section. Then $\mathfrak{g}$ is a real reductive Lie algebra with
$\dim\mathfrak{c}(\mathfrak{g})\leq1$.
\end{lemma}
\begin{proof}
Denote $\mathfrak{r}$ and $\mathfrak{n}$ the radical and nilradical of $\mathfrak{g}$ respectively. Let $\mathfrak{s}
=\mathfrak{s}_c\oplus\mathfrak{s}_{nc}$ be a Levi subalgebra
of $\mathfrak{g}$, in which $\mathfrak{s}_c$ and $\mathfrak{s}_{nc}$ are the compact and non-compact parts of
$\mathfrak{s}$ respectively.

Firstly, we prove

{\bf Claim 1:}
$\dim\mathfrak{n}\leq 1$ and $[\mathfrak{s},\mathfrak{n}]=0$.

We assume conversely that $\dim\mathfrak{n}>1$.
Let $N=\mathrm{Nil}(G)$ be the connected normal subgroup of $G$ generated by $\mathfrak{n}$. Then $G'=NH$ is connected subgroup of $G$ containing $H$. By Lemma \ref{lemma-flag-curvature-totally-geodesic} and Lemma \ref{lemma-go-total-geod-tech}, $(G'/H,F|_{G'/H})$ is a
flagwise positively and non-negatively curved
totally geodesic submanifold in $(G/H,F)$. On the other hand,
$(G'/H,F_{G'/H})$ is locally isometric to the nilpotent Lie group $N$ endowed with a left invariant Finsler metric $F''$. If $N$ is not Abelian, the flag curvature of $(N,F'')$ is negative somewhere (see Theorem 1.1 in \cite{Huang2015-2}), which contradicts the non-negatively curved condition. If $N$ is Abelian, $(N,F'')$ is flat, which contradicts the (FP) condition.
To summarize, this argument proves $\dim\mathfrak{n}\leq1$.

We continue to prove $[\mathfrak{s},\mathfrak{n}]=0$.

If $\dim\mathfrak{n}=0$, there is nothing to prove. Otherwise,
we consider
the restriction of $\mathrm{ad}(\mathfrak{s})$-actions from $\mathfrak{n}$ to $\mathfrak{n}$. It
defines Lie algebra endomorphism from $\mathfrak{s}$ to
$\mathbb{R}$, which must vanish. So we get $[\mathfrak{s},\mathfrak{n}]=0$.

This ends the proof of
Claim 1.

Secondly, we prove

{\bf Claim 2:} $\dim\mathfrak{r}\leq 2$ and $[\mathfrak{s},\mathfrak{r}]=0$.

Assume conversely that $\dim\mathfrak{r}>2$. It implies $\mathfrak{n}\neq0$, so by Claim 1, $\dim\mathfrak{n}=1$.
The restriction of $\mathrm{ad}(\mathfrak{r})$ from $\mathfrak{n}$
to $\mathfrak{n}$ are real scalar changes. So
there exists
a nonzero vector $u\in\mathfrak{r}\backslash\mathfrak{n}$, such
that $[u,\mathfrak{n}]=0$. Since $[\mathfrak{g},\mathfrak{r}]
\subset\mathfrak{n}$, $\mathfrak{n}'=\mathbb{R}u+\mathfrak{n}$
is a nilpotent ideal of $\mathfrak{g}$ bigger than $\mathfrak{n}$.
This is a contradiction.
To summarize, we have proved $\dim\mathfrak{r}\leq2$.

We continue to prove $[\mathfrak{s},\mathfrak{r}]=0$.
This has already been proved when $\mathfrak{r}=\mathfrak{n}$.
So we only need to consider the case that
$\dim\mathfrak{r}=2$ and $\dim\mathfrak{n}=1$.
Since $[\mathfrak{s},\mathfrak{r}]\subset\mathfrak{n}$ and
$[\mathfrak{s},\mathfrak{n}]=0$,
the restriction of $\mathrm{ad}(\mathfrak{s})$-actions from
$\mathfrak{r}$ to $\mathfrak{r}$ defines a Lie algebra endomorphism from the semi-simple Lie algebra $\mathfrak{s}$
to $\mathbb{R}$, which must vanish. So we get $[\mathfrak{s},\mathfrak{r}]=0$, and end the proof of Claim 2.

Finally, we prove

{\bf Claim 3:} $\dim\mathfrak{r}\leq 1$.

Assume conversely that $\dim\mathfrak{r}=2$, then we also have
$\dim\mathfrak{n}=1$. Since $H$ is compact and $\mathrm{Ad}(H)$
preserves the one-dimensional $\mathfrak{n}$, we must have
$[\mathfrak{h},\mathfrak{n}]=0$, and then $\mathfrak{h}\subset\mathfrak{s}+\mathfrak{n}$. By Claim 2,
we have the Lie algebra direct sum
$\mathfrak{g}=\mathfrak{s}\oplus\mathfrak{r}=\mathfrak{s}_c
\oplus\mathfrak{s}_{nc}\oplus\mathfrak{r}.$
It is obvious to see that $B_\mathfrak{g}(\mathfrak{h},\mathfrak{r})=0$, i.e.
$\mathfrak{r}\subset\mathfrak{m}$ has a zero intersection with
$\mathfrak{h}$.

Let $R=\mathrm{Rad}(G)$ be the closed normal solvable subgroup of $G$ generated by $\mathfrak{r}$. Then $G'=RH$ is a connected Lie
group containing $H$. By Lemma \ref{lemma-go-total-geod-tech},
$(G'/H,F|_{G'/H})$ is a $G'$-g.o.~Finsler space. Since $R$ and $H$
commutes, $H$ is normal in $G'$, $(G'/H,F|_{G'/H})$ is $G'/H$-g.o.~ as well, i.e., it is locally isometric to the Lie group $R$ endowed
with a left invariant $R$-g.o.~Finsler metric. By Lemma \ref{lemma-left-invariant-go}, $\mathfrak{r}$ is compact. But this
is not true.

To summarize, we get $\dim\mathfrak{r}\leq 1$ by contradiction,
and ends the proof of Claim 3.

Summarizing Claim 2 and Claim 3, Lemma \ref{step-1} follows easily.
\end{proof}

We remark that,
with the extra curvature conditions,
Lemma \ref{step-1} strengthens the results
in \cite{XDY-preprint}, where we have proved
$\mathfrak{n}$ has a step size at most two,
and $[\mathfrak{s}_{nc},\mathfrak{r}]=0$.
\subsection{Rank inequality}
We keep all assumptions and notations in this section.
By Lemma \ref{step-1}, $\mathfrak{g}$ is a real reductive Lie algebra with $\dim\mathfrak{c}(\mathfrak{g})\leq1$.

In this subsection, we first prove the rank inequality.

\begin{lemma}\label{step-2}
Keep all assumptions and notations in this section, then
$\mathrm{rk}\mathfrak{g}\leq\mathrm{rk}\mathfrak{h}+1$.
\end{lemma}

\begin{proof}
Let $T_H$ be the maximal torus subgroup in $H$, and $G'$ the identity component of the centralizer $C_G(T_H)$ of $T_H$ in $G$.
By Lemma \ref{lemma-go-fix-point-tech}, the fixed point set  $\mathrm{Fix}_o(T_H,G/H)$ is a totally geodesic submanifold, and $(\mathrm{Fix}_o(T_H,G/H),F|_{\mathrm{Fix}_o(T_H,G/H)})$
is $G'$-g.o., and as we have observed, g.o.~with respect to the quotient group $G''=G'/T_H$ as well.
Since $G'\cap H$ has the same Lie algebra as $T_H$, i.e.
$(G'\cap H)/T_H$ is a finite group, the restriction of $F$ to
$\mathrm{Fix}_o(T_H,G/H)=G'/G'\cap H=G''/((G'\cap H)/T_H)$
is locally isometric to a $G''$-g.o.~left invariant Finsler metric $F''$
on $G''$. By Lemma \ref{lemma-left-invariant-go}, $\mathfrak{g}''=
\mathrm{Lie}(G'')$ is compact and $F''$ is bi-invariant. Any maximal torus $T''\subset G''$, which can be viewed as the fixed point set of $T''\subset \Delta G''$ in $G''=G''\times G''/\Delta G''$, is a totally geodesic flat Finsler submanifold in $(G'',F'')$.

Assume conversely that $\mathrm{rk}\mathfrak{g}>\mathrm{rk}(\mathfrak{h})+1$,
then $\dim T''=\mathrm{rk}\mathfrak{g}''=\mathrm{rk}\mathfrak{g}-
\mathrm{rk}\mathfrak{h}>1$. By Lemma \ref{lemma-flag-curvature-totally-geodesic},
$(T'',F|_{T''})$ satisfies the (FP) condition. This is
a contradiction because $(T'',F|_{T''})$ is flat.

To summarize, this ends the proof of Lemma \ref{step-2}.
\end{proof}

Then we apply Lemma \ref{step-2} to prove the following lemma.

\begin{lemma}\label{step-3}
We keep all assumptions and notations in this section. If
$\mathfrak{c}(\mathfrak{g})\neq0$, then $\mathfrak{h}$ is
not contained in $[\mathfrak{g},\mathfrak{g}]$.
\end{lemma}

\begin{proof}
We denote $\mathfrak{g}=\mathfrak{c}(\mathfrak{g})\oplus
\mathfrak{g}_c\oplus\mathfrak{g}_{nc}$ the Lie algebra direct
sum previously mentioned, then $[\mathfrak{g},\mathfrak{g}]=
\mathfrak{g}_c\oplus\mathfrak{g}_{nc}$.

Assume conversely that $\mathfrak{c}(\mathfrak{g})\neq0$
and $\mathfrak{h}\subset[\mathfrak{g},\mathfrak{g}]$. By
the rank inequality, $\dim\mathfrak{c}(\mathfrak{g})=1$ and
$\mathrm{rk}\mathfrak{h}=\mathrm{rk}\mathfrak{g}_c+
\mathrm{rk}\mathfrak{g}_{nc}$, there exists Cartan subalgebras
$\mathfrak{t}_c$ and $\mathfrak{t}_{nc}$ in $\mathfrak{g}_c$ and
$\mathfrak{g}_{nc}$ respectively, such that $\mathfrak{t}_c+\mathfrak{t}_{nc}$ is a Cartan subalgebra of $\mathfrak{h}$.

We can also find  a  maximal compactly imbedded subalgebra
$\mathfrak{k}_{nc}$ in $\mathfrak{g}_{nc}$ such that $\mathfrak{h}\subset\mathfrak{g}_c+\mathfrak{k}_{nc}$. We have
Cartan decompositions $\mathfrak{g}_{nc}=\mathfrak{k}_{nc}+\mathfrak{p}_{nc}$ and
$\mathfrak{g}_{nc}^u=\mathfrak{k}_{nc}+\sqrt{-1}\mathfrak{p}_{nc}$.
With respect to $\mathfrak{t}_c$ and $\mathfrak{t}_{nc}$, we have
the following root plane decompositions
\begin{eqnarray}
\mathfrak{g}_c&=&\mathfrak{t}_c+\sum_{\alpha\in\Delta_c}
\mathfrak{g}_{c;\pm\alpha},\mbox{ and}\nonumber\\
\mathfrak{g}_{nc}^u&=&\mathfrak{t}_{nc}+\sum_{\alpha\in\Delta^u_{nc}}
\mathfrak{g}^u_{nc;\pm\alpha},\nonumber
\end{eqnarray}
where $\Delta_c\subset\mathfrak{t}_c^*$ and
$\Delta_{nc}^u\subset\mathfrak{t}_{nc}^*$ are root systems,
and $\mathfrak{g}_{c;\pm\alpha}$ and $\mathfrak{g}_{nc;\pm\alpha}^u$ are root planes.

We first prove

{\bf Claim 1:} $\mathfrak{g}_{nc}=0$.

Assume conversely that $\mathfrak{g}_{nc}\neq0$,
then there exists a root $\alpha\in
\Delta_{nc}^u\in\mathfrak{t}_{nc}^*$, such that the
root plane $\mathfrak{g}_{nc;\pm\alpha}^u$ is contained in
$\sqrt{-1}\mathfrak{p}_{nc}$. Let $T'$ be the sub-torus
in $H$ generated by $\mathfrak{t}'=\mathfrak{t}_c+\mathrm{ker}\alpha$.
By Lemma \ref{lemma-flag-curvature-totally-geodesic} and Lemma \ref{lemma-go-fix-point-tech},
the restriction of $F$ to the fixed point set $\mathrm{Fix}_o(T',G/H)$
is locally isometric to a non-negatively curved $SL(2,\mathbb{R})\times\mathbb{R}$-g.o.~Finsler metric on
$(SL(2,\mathbb{R})\times \mathbb{R})/(SO(2)\times\{0\})$. But this is impossible by Theorem \ref{thm-preparation}.

This ends the proof of Claim 1.

Then we prove

{\bf Claim 2:} $\mathfrak{g}_c\subset\mathfrak{h}$.

Assume conversely that $\mathfrak{g}_c$ is not contained in
$\mathfrak{h}$. Then we can find a root $\alpha\in\Delta_c$,
such that the root plane $\mathfrak{g}_{c;\pm\alpha}$ is not
contained in $\mathfrak{h}$. Let $T'$ be the sub-torus
in $H$ generated by $\mathfrak{t}'=\mathrm{ker}\alpha$.
By Lemma \ref{lemma-flag-curvature-totally-geodesic} and Lemma \ref{lemma-go-fix-point-tech},
the restriction of $F$ to the fixed point set $\mathrm{Fix}_o(T',G/H)$
is locally isometric to a
homogeneous Finsler metric $F'$ on
$(SO(3)\times\mathbb{R})/(SO(2)\times\{0\})$ satisfying the
(FP) condition. Since $((SO(3)\times\mathbb{R})/(SO(2)\times\{0\}),F')$
is a rank two globally symmetric Finsler space, by Lemma \ref{prop-symmetric-space-non-compact}, it can not satisfy
the (FP) condition. This is a contradiction ends the proof of
Claim 2.

Summarizing the two claims, we see that $\dim G/H=1$, contradicting the dimension assumption $\dim G/H>1$.

This ends the proof of Lemma \ref{step-3}.
\end{proof}

\subsection{Proof of Theorem \ref{main-thm} and its corollary}

Now we summarize above argument to prove Theorem \ref{main-thm}.

{\bf Proof of Theorem \ref{main-thm}.}

Firstly, we prove the the statement for the compactness.

Let $(M,F)$ with $\dim M>1$ be a g.o.~Finsler space which has non-negative flag curvature and satisfies the (FP) condition. We denote $G'=I_0(M,F)$
its connected isometry group. Then by Lemma \ref{step-1}, $\mathfrak{g}'=\mathrm{Lie}(G')$ is a real reductive Lie algebra
with $\dim\mathfrak{c}(\mathfrak{g'})\leq1$. If we present $M$ as
$M=G'/H'$, then the isotropy subgroup $H'$ at $o=eH'$ is compact.
By Lemma \ref{step-2}, we have the rank inequality
$\mathrm{rk}\mathfrak{g'}\leq\mathrm{rk}\mathfrak{h'}+1$.
So by Theorem \ref{thm-preparation}, we see that $\mathfrak{g'}$ is compact. If $\mathfrak{g'}$ is semi-simple,
$G$ is compact itself. Otherwise
 $\mathfrak{g}'=\mathfrak{g}''\oplus\mathbb{R}$,
in which $\mathfrak{g}''$ is a compact semi-simple subalgebra.
By Lemma \ref{step-3}, $\mathfrak{h}'$ is not contained
in $\mathfrak{g}''$, so the compact subgroup $G''$ generated by
$\mathfrak{g}''$ acts transitively on $M$.
To summarize, in both cases, $M$ admits the transitive action of
a compact Lie group, so $M$ is compact.

Nextly, we prove the statement for the rank inequality. Let $G$ be a quasi-compact Lie group, with the compact Lie algebra $\mathfrak{g}$, which acts
isometrically and transitively on $(M,F)$, and $H$ the corresponding
isotropy subgroup, which Lie algebra $\mathfrak{h}$ is also compact.
We may further assume that $G$ acts effectively, i.e., $\mathfrak{g}\subset\mathfrak{g}'$ and $\mathfrak{h}=\mathfrak{h}'\cap\mathfrak{g}$,
because changing $G$ and $H$
with their images in $G'=I_0(M,F)$ does not change the value of
$\mathrm{rk}\mathfrak{g}-\mathrm{rk}\mathfrak{h}$.

We can find Cartan subalgebras $\mathfrak{t}'_1$ and $\mathfrak{t}'_2$ of $\mathfrak{g}'$, such that
$\mathfrak{t}'_1\cap\mathfrak{h}'$ and $\mathfrak{t}'_2\cap\mathfrak{g}$ are Cartan subalgebras
of $\mathfrak{h}'$ and $\mathfrak{g}$ respectively. We can find
an element $g\in G'$, such that $\mathrm{Ad}(g)\mathfrak{t}'_1=
\mathfrak{t}'_2$. Since $G$ acts transitively on $M=G'/H'$, we
can find $g_1\in H'$ and $g_2\in G$, such that $g=g_2^{-1}g_1$.
Then for the Cartan subalgebra $\mathfrak{t}=\mathrm{Ad}(g_1)\mathfrak{t}'_1
=\mathrm{Ad}(g_2)\mathfrak{t}'_2$,
$\mathfrak{t}\cap\mathfrak{h}'$ and $\mathfrak{t}\cap\mathfrak{g}$ are Cartan subalgebras of
$\mathfrak{h}'$ and $\mathfrak{g}$ respectively.
By Lemma \ref{step-2}, $\dim\mathfrak{t}'-\dim\mathfrak{t}'\cap\mathfrak{h}'=
\mathrm{rk}\mathfrak{g}'-\mathrm{rk}\mathfrak{h}'\leq1$,
so we have
\begin{eqnarray}
\mathrm{rk}\mathfrak{g}-\mathrm{rk}\mathfrak{h}
&\leq&\dim\mathfrak{t}'\cap\mathfrak{g}-
\dim\mathfrak{t}'\cap\mathfrak{h}
\leq\dim\mathfrak{t}'-\dim\mathfrak{t}'\cap\mathfrak{h}'\nonumber\\
&=&\mathrm{rk}\mathfrak{g}'-\mathrm{rk}\mathfrak{h}'\leq1,
\nonumber
\end{eqnarray}
which proves the rank inequality for $G$ and $H$.

This ends the proof of Theorem \ref{main-thm}.
\ \rule{0.5em}{0.5em}

Notice that for any compact connected Lie group $G$ and its closed subgroup $H$, $G/H$ has a positive Euler number iff
$\mathrm{rk}\mathfrak{g}=\mathrm{rk}\mathfrak{h}$.
So for $G$ and $H$ in Theorem \ref{main-thm}, $\mathrm{rk}\mathfrak{g}-\mathrm{rk}\mathfrak{h}$ equals $0$ or $1$, when $\dim M$ is even or odd respectively.

As an application for Theorem \ref{main-thm}, we discuss an even dimensional g.o.~Finsler space $(M,F)$ which has non-negative flag curvature and satisfies
the (FP) condition.

By Theorem \ref{main-thm}, we can find a compact Lie group $G\subset I_0(M,F)$ which acts transitively on $M$. If we denote
$M=G/H$, then we have $\mathrm{rk}\mathfrak{g}=\mathrm{rk}\mathfrak{h}$
by the rank inequality in Theorem \ref{main-thm} and the assumption that $\dim M$ is even.

Let $\mathfrak{t}$ be a Cartan subalgebra of $\mathfrak{g}$ contained in $\mathfrak{h}$. Then we have the root plane decompositions of $\mathfrak{g}$ and $\mathfrak{h}$ with respect to $\mathfrak{t}$. A root $\alpha$ of $\mathfrak{g}$ belongs to
$\mathfrak{h}$ iff $\mathfrak{g}_{\pm\alpha}\subset\mathfrak{h}$,
i.e., the root plane $\mathfrak{g}_{\pm\alpha}$ of $\mathfrak{g}$ is also a root plane of $\mathfrak{h}$.

\begin{lemma}\label{lemma-last}
Keep all above assumptions and notations for an even dimensional g.o. Finsler
space $(M,F)$ which has non-negative flag curvature and
satisfies the (FP) condition. Then there do not exist a pair
of linearly independent roots $\alpha$ and $\beta$, such that
they are not roots of $\mathfrak{h}$ and $\alpha\pm\beta$ are
not roots of $\mathfrak{g}$.
\end{lemma}
\begin{proof}
Assume conversely that there exists such a pair of roots $\alpha$
and $\beta$ as indicated in the lemma. Denote $T'$ the sub-torus
generated by $\mathfrak{t}'=\mathrm{ker}\alpha \cap \mathrm{ker}\beta$. By Lemma \ref{lemma-flag-curvature-totally-geodesic}, then the restriction of $F$ to
$\mathrm{Fix}_o(T',M)$ is locally isometric to an $SO(3)\times SO(3)$-invariant metric $F'$ on $M'=S^2\times S^2=(SO(3)\times SO(3))/(SO(2)\times SO(2))$ which has non-negative flag curvature
and satisfies the (FP) condition. But $(M',F')$ is a rank two globally symmetric Finsler space. By Lemma \ref{prop-symmetric-space-non-compact}, it can not satisfy
the (FP) condition. This is a contradiction, which ends the proof
of this lemma.
\end{proof}

By Lemma \ref{lemma-last} and the algebraic argument in \cite{Wa1972,XDHH}, we see that $G/H$ belongs to the classification list of even dimensional positively curved homogeneous
Riemannian or Finsler manifolds, i.e.,
\begin{eqnarray}
& &SO(2n+1)/SO(2n), G_2/SU(3), SU(n)/S(U(n-1)U(1)), \nonumber\\
& &Sp(n)/Sp(n-1)U(1),Sp(n)/Sp(n-1)Sp(1), F_4/Spin(9), \label{classification-list-1}\\
& &SU(3)/T^2, Sp(3)/Sp(1)^3, F_4/Spin(8).\label{classification-list-2}
\end{eqnarray}

On each coset spaces in (\ref{classification-list-1}), we can
find positively curved $G$-normal homogeneous Riemannian metrics.
However, the three Wallach spaces in (\ref{classification-list-2})
do not admit positively curved $G$-g.o.~homogeneous Riemannian metrics \cite{AW2017,Ber61}. Unlike the rareness of g.o.~Riemannian metrics on the Wallach spaces, there are too many Finsler ones to be completely classified.
Until now, we can not determine if the three Wallach spaces
admit g.o.~Finsler metrics metrics which are either
positively curved or satisfy the requirement in Theorem \ref{main-thm}.

At the end, we summarize above application to the following
corollary.

\begin{corollary}\label{last corollary}
Any connected even dimensional g.o.~Finsler space which has
non-negative flag curvature and satisfies the
(FP) condition must be a compact smooth coset space
which admits positively curved homogeneous Riemannian
or Finsler metrics.
\end{corollary}

{\bf Acknowledgements.} The author sincerely thanks
Fernando Galaz-Garcia,
Yuri G. Nikonorov and Wolfgang Ziller for helpful discussions.
This paper is supported by National Natural Science Foundation of China (No. 11821101, No. 11771331), Beijing Natural Science Foundation
(No. 00719210010001, No. 1182006), Capacity Building for Sci-Tech  Innovation -- Fundamental Scientific Research Funds (No. KM201910028021).

\vspace{5mm}
\end{document}